\documentclass{amsart}
\usepackage{amsfonts}
\usepackage{amsmath,amssymb}
\usepackage{amsthm}
\usepackage{amscd}
\usepackage{graphics}
\usepackage{graphicx}

\theoremstyle{remark}{
\newtheorem{Def}{{\rm Definition}}
\newtheorem{Ex}{{\rm Example}}

\newtheorem{Prob}{{\rm Problem}}

}
\theoremstyle{plain}
{

\newtheorem{Prop}{Proposition}
\newtheorem{Thm}{Theorem}

}

\begin{document}
\title[Labeled versions of Poincar\'e-Reeb graphs.]{Arrangements of circles, the regions surrounded by them and labeled Poincar\'e-Reeb graphs}
\author{Naoki kitazawa}

\address{Institute of Mathematics for Industry, Kyushu University, 744 Motooka, Nishi-ku Fukuoka 819-0395, Japan\\
 TEL (Office): +81-92-802-4402 \\
 FAX (Office): +81-92-802-4405 \\
}
\email{n-kitazawa@imi.kyushu-u.ac.jp, naokikitazawa.formath@gmail.com}
\urladdr{https://naokikitazawa.github.io/NaokiKitazawa.html}
\maketitle
\begin{abstract}
We are interested in arrangements of circles and the regions surrounded by them. {\it Poincar\'e-Reeb graphs} have been fundamental and strong tools in studying shapes of regions surrounded by real algebraic curves, since around 2020. They are natural graphs the regions naturally collapse to and were first formulated by Sorea with several researchers. Studying shapes of such regions is one of fundamental studies in real algebraic geometry and combinatorics for example. This is surprisingly new and recently developing. 

Our study introduces labels on vertices and edges of such graphs encoding information of the circles where we concentrate on regions surrounded by circles. The author studied local changes of Poincar\'e-Reeb graphs by addition of circles under certain rules before and we discuss changes of new types. The author has started related studies motivated by singularity theory of real algebraic maps and found first that our regions are the images of natural real algebraic maps, generalizing natural projections of spheres.

\end{abstract}
\section{Introduction.}
\label{sec:1}

Regions surrounded by real algebraic curves are natural objects in real algebraic geometry and combinatorics for example. One of fundamental and important studies is to study shapes of them. 
Such a study is also new and recently developing. \cite{bodinpopescupampusorea, sorea1, sorea2} are of related pioneering studies.

We are also interested in shapes of natural graphs these regions collapse to. 
Especially, the author is also interested in regions surrounded by circles and local changes of graphs according to addition of circles under certain rules. Related studies are \cite{kitazawa3, kitazawa4} for example. The author is also interested in singularity theory of differentiable maps and applications to differential topology of manifolds, especially, interested in construction of explicit smooth maps and real algebraic maps. More explicitly, one of related interest lies in explicitly constructing functions and maps generalizing natural projections of spheres ({\it the canonical projections of the unit spheres}), and (the closures of) our regions are the images of natural real algebraic maps. Related pioneering studies are \cite{kitazawa1, kitazawa2} and followed by \cite{kitazawa3, kitazawa4}.

\subsection{Notation on topological spaces and manifolds.}
This is for fundamental notation on topological spaces and manifolds. We use ${\mathbb{R}}^n$ for the $n$-dimensional Euclidean space, equipped with the standard Euclidean metric. 

We use $||x|| \geq 0$ for the distance between $x \in {\mathbb{R}}^n$ and the origin $0 \in {\mathbb{R}}^n$ under the metric. This is also seen as a real vector space (Hilbert space) of dimension $n$.
We also use $\mathbb{R}:={\mathbb{R}}^1$.

We use ${\pi}_{m,n,i}:{\mathbb{R}}^m \rightarrow {\mathbb{R}}^n$ with $m>n \geq 1$ for the canonical projection (let $x=(x_1,x_2) \in {\mathbb{R}}^{n} \times {\mathbb{R}}^{m-n}={\mathbb{R}}^m$ and ${\pi}_{m,n,i}(x)=x_i$ where $i=1,2$). The set $D^k:=\{x \in {\mathbb{R}}^k \mid ||x|| \leq 1\}$ is the $k$-dimensional unit disk and $S^k:=\{x \in {\mathbb{R}}^{k+1} \mid ||x||=1\}$ is the $k$-dimensional unit sphere.

For a subspace $Y \subset X$ of a topological space $X$, we use $\overline{Y}$ for its closure in $X$ and $Y^{\circ}$ for the interior in $X$: unless otherwise stated we do not refer to the outer space $X$. 

We use $\partial X$ for the boundary of a (topological) manifold $X$.  
\subsection{Our arrangements of circles.}

\label{subsec:1.2}
This reviews some of \cite{kitazawa3, kitazawa4}.

\begin{Def}
\label{def:1}
In our paper, a {\it circle} {\it centered at} $x_0=(x_{0,1},x_{0,2}) \in {\mathbb{R}}^2$ and of {\it radius} $r>0$ means a set of the form $S _{x_0,r}:=\{x \mid ||x-x_0||=r\} \subset {\mathbb{R}}^2$. We call the points $(x_{0,1}+r\cos \frac{j\pi}{4},x_{0,2}+r\sin \frac{j\pi}{4})$ {\it $(j,\frac{\pi}{4})$-poles} of the circle  {\rm (}$j=0,1,2,3,4,5,6,7${\rm )}. These eight points divide the circle into eight connected arcs in the circle. We call each arc connecting $(x_{0,1}+r\cos \frac{j\pi}{4},x_{0,2}+r\sin \frac{j\pi}{4})$ and $(x_{0,1}+r\cos \frac{(j+1)\pi}{4},x_{0,2}+r\sin \frac{(j+1)\pi}{4})$ the {\it $(j,j+1) \frac{\pi}{4}$-arc} of the circle {\rm (}$j=0,1,2,3,4,5,6,7${\rm )}. 

\end{Def}


\begin{Def}
\label{def:2}
We call a pair $(\mathcal{S},D_{\mathcal{S}})$ of a set $\mathcal{S}:=\{S_{x_{j},r_{j}}\}$ of circles and a bounded connected component $D_{\mathcal{S}} \subset {\mathbb{R}}^2$ of the complementary set ${\mathbb{R}}^2-{\bigcup}_{S_{x_{j},r_{j}} \in \mathcal{S}} S_{x_{j},r_{j}}$ a {\it normally inductive arrangement} ({\it NI arrangement}) {\it of circles} if we have $\mathcal{S}$ and $D_{\mathcal{S}}$ inductively as follows.
\begin{enumerate}
\item \label{def:3.1} First choose a non-empty set ${\mathcal{S}}_0$ of mutually disjoint circles $S_{x_{j,0},r_{j,0}} \in {\mathcal{S}}_0$ such that the boundary $\partial \overline{D_{{\mathcal{S}}_0}}$ of the closure $\overline{D_{{\mathcal{S}}_0}}$ of the bounded connected component $D_{{\mathcal{S}}_0} \subset {\mathbb{R}}^2$ of the complementary set ${\mathbb{R}}^2-{\bigcup}_{S_{x_{j,0},r_{j,0}} \in 
{\mathcal{S}}_0} S_{x_{j,0},r_{j,0}}$ is the disjoint union ${\sqcup}_{S_{x_{j,0},r_{j,0}} \in {\mathcal{S}}_0} S_{x_{j,0},r_{j,0}}$.
\item \label{def:3.2} Put ${\mathcal{S}}:={\mathcal{S}}_0$.

We choose a new circle $S_{x_{j^{\prime}},r_{j^{\prime}}}$ of radius $r_{j^{\prime}}>0$ centered at $x_{j^{\prime}} \in {\mathbb{R}}^2$ and intersecting at least one circle from the set ${\mathcal{S}}$. We also choose the circle obeying the following rule: two distinct circles $S_{x_{j},r_{j}}$ and $S_{x_{j^{\prime}},r_{j^{\prime}}}$ always intersect at points in $\overline{D_{\mathcal{S}}}$ according to the rule that for each point $p_{j,j^{\prime}}$ of such points, the sum of the tangent vector spaces of $S_{x_{j},r_{j}}$  and $S_{x_{j^{\prime}},r_{j^{\prime}}}$ at $p_{j,j^{\prime}}$ coincides with the tangent vector space of ${\mathbb{R}}^2$ at $p_{j,j^{\prime}}$, and three distinct circles $S_{x_{j_1},r_{j_1}}$, $S_{x_{j_2},r_{j_2}}$, and $S_{x_{j^{\prime}},r_{j^{\prime}}}$ do not intersect at any point in $\overline{D_{\mathcal{S}}}$. 
We define the new set ${\mathcal{S}}^{\prime}$ as the disjoint union $\mathcal{S} \sqcup \{S_{x_{j^{\prime}},r_{j^{\prime}}}\}$ and the new region $D_{{\mathcal{S}}^{\prime}}$ as the intersection $D_{\mathcal{S}} \bigcap {E_{x_{j^{\prime}},r_{j^{\prime}}}}$ where $E_{x_{j^{\prime}},r_{j^{\prime}}}:={D_{x_{j^{\prime}},r_{j^{\prime}}}}^{\circ}$ or $E_{x_{j^{\prime}},r_{j^{\prime}}}:={\mathbb{R}}^2-D_{x_{j^{\prime}},r_{j^{\prime}}}$ with $D_{x_{j^{\prime}},r_{j^{\prime}}} \subset {\mathbb{R}}^2$, denoting the closed disk whose boundary is $S_{x_{j^{\prime}},r_{j^{\prime}}}$. We also pose a constraint that the intersections $\overline{D_{{\mathcal{S}}}^{\prime}} \bigcap S _{x_j,r_j}$ and $\overline{D_{{\mathcal{S}}}^{\prime}} \bigcap S _{x_{j^{\prime}},r_{j^{\prime}}}$ are non-empty sets for all circles $S _{x_j,r_j}, S_{x_{j^{\prime}},r_{j^{\prime}}} \in {\mathcal{S}}^{\prime}$ here.

We put $(\mathcal{S},D_{\mathcal{S}}):=({\mathcal{S}}^{\prime},D_{{\mathcal{S}}^{\prime}})$. 
We do this procedure inductively.
\end{enumerate}
\end{Def}

\subsection{Organization of our paper and our main work.}
In the next section, we introduce fundamental terminologies and notation on topological spaces decomposed into cells, smooth maps between smooth manifolds, and (V-di)graphs (with their vertices and edges).

In the third section, we review Poincar\'e-Reeb graphs of regions $D_{\mathcal{S}}$ surrounded by all circles of $\mathcal{S}$. They are graphs the regions naturally collapse to respecting the projections ${\pi}_{2,1,i}$. 
They are the spaces of all connected components of all preimages of single points for the restrictions of the projections ${\pi}_{2,1,i}$ and the quotient spaces of the closures $\overline{D_{\mathcal{S}}}$ of the regions topologized with the corresponding quotient topologies. We also introduce one of our main result in the second subsection of the third section (Theorem \ref{thm:1}). More precisely, we assign labels to vertices and edges respecting the preimages for ${\pi}_{2,1,i}$, curves there which are also in some circles of $\mathcal{S}$, and $(j,j+1)\frac{\pi}{4}$-arcs of these circles of $\mathcal{S}$ which are also in the preimages. We also see some fundamental properties.

In the fourth section, we discuss a certain rule of addition of circles to given NI arrangements of circles and investigate the local changes of our Poincar\'e-Reeb graphs. We present our related new result as Theorems \ref{thm:2} and \ref{thm:3}. We also review related previous studies and compare our new result to previously obtained related results.

The fifth section is devoted to remarks.

\section{Topological spaces decomposed into cells, smooth maps and graphs.}
\label{sec:2}

The dimension $\dim X$ of a non-empty topological space $X$ is the unique non-negative integer and gives a topological invariant for such spaces. (Topological) manifolds, graphs, polyhedra, and CW complexes are of the class of such spaces.

For a differentiable manifold $X$, let $T_p X$ denote the tangent vector space at $p \in X$.
For a differentiable map $c:X \rightarrow Y$ between differentiable manifolds $X$ and $Y$, the differential ${dc}_p:T_pX \rightarrow T_{c(p)} Y$ at $p \in X$ is defined as a linear map. The point $p$ is a {\it singular} point of $c$ if the rank of ${dc}_p$ drops. If these manifolds are seen as smooth ones or ones of the class $C^{\infty}$, the space $C^{\infty}(X,Y)$ of all smooth maps between them are topologized with the so-called {\it Whitney $C^{\infty}$ topology}. This topology is, roughly, a topology making two smooth maps close if and only if their values and the values of their derivations at each point are close.  We do not assume related knowledge. For related singularity theory of differentiable (smooth) maps, consult \cite{golubitskyguillemin} for example.

A {\it diffeomorphism} means a smooth map which is also a homeomorphism having no singular point. A {\it diffeomorphism on a smooth manifold} $X$ means a diffeomorphism from $X$ onto itself. The {\it diffeomorphism group} of $X$ is the group of all diffeomorphisms on $X$. A {\it smooth} bundle means a bundle whose fiber is a smooth manifold and whose structure group is a subgroup of the diffeomorphism group of the fiber.

A graph is a CW complex of dimension $1$. $0$-cells of the graph are {\it vertices} (of the graph). $1$-cells of the graph are {\it edges} (of the graph). Note that edges are open subsets of the graph here. The set of all vertices (edges) of the graph is called the {\it vertex} (resp. {\it edge}) {\it set} (of the graph).
A {\it digraph} is a graph all of whose edges are oriented. A {\it V-digraph} $(K,l_K)$ is a pair of a digraph $K$ and a map $l_K:V_K \rightarrow P$ from the vertex set $V_K$ of $K$ into a totally ordered set $P$ such that the edges are oriented according to the values $l_K(v)$. In other words the relation $l_K(v_1) \neq l_K(v_2)$ holds if there exists an edge $e$ connecting $v_1$ and $v_2$ and $e$ is an oriented edge departing from $v_1$ and entering $v_2$ if $l_K(v_1)<l_K(v_2)$: here "$<$" is for the order on $P$. The {\it degree} of a vertex of a graph means the number of edges of the graph containing the vertex.
\begin{Def}
Two graphs $K_1$ and $K_2$ are said to be {\it isomorphic} if a piecewise smooth homeomorphism $\phi:K_1 \rightarrow K_2$ mapping the vertex set of $K_1$ onto that of $K_2$ exists. This map is an {\it isomorphism} between the graphs. Two digraphs are said to be {\it isomorphic} if an isomorphism of the (underlying) graphs preserving the orientations of the edges exists. This is an {\it isomorphism} between the digraphs.
Two V-digraphs are said to be {\it isomorphic} if an isomorphism of the (underlying) digraphs preserving the orders of the values of $l_K$ exists. This is an {\it isomorphism} between the V-digraphs.
\end{Def}

\section{Poincar\'e-Reeb V-digraphs with labels.}
\label{sec:3}
\subsection{Reviewing the Poincar\'e-Reeb V-digraph of $(\mathcal{S},D_{\mathcal{S}})$ for ${\pi}_{2,1,i}$.}
We review the Poincar\'e-Reeb (V-di)graph of $(\mathcal{S},D_{\mathcal{S}})$ for ${\pi}_{2,1,i}$. The underlying set is the set of all connected components of all preimages of all single points for ${\pi}_{2,1,i} {\mid}_{\overline{D_{\mathcal{S}}}}$. We can topologize this by the quotient topology. The underlying graph is the graph whose vertex set consists of all points (connected components of the preimages) containing $(j,\frac{\pi}{4})$-poles with $j=0,4$ (resp. $(j,\frac{\pi}{4})$-poles with $j=2,6$) of some circles of $\mathcal{S}$ or points in exactly two distinct circles of $\mathcal{S}$ in the case $i=1$ (resp. $i=2$). 

We review essential parts in proving that this is a graph where more precise exposition is left to \cite{kitazawa3}. We do not need to understand related arguments well. Over interiors of edges, the projection ${\pi}_{2,1,i} {\mid}_{D_{\mathcal{S}}}$ yields the structure of a trivial smooth bundle whose fiber is diffeomorphic to a closed interval $D^1$. So-called (relative) Ehresmann fibration theorem is applied. Local structures around vertices can be easily checked.

\cite{saeki1, saeki2} also explain this fact from some general viewpoint.

We use $q_{D_{\mathcal{S}},i}:\overline{D_{\mathcal{S}}} \rightarrow G_{D_{\mathcal{S}},i}$ for the quotient map onto the graph $G_{D_{\mathcal{S}},i}$.
This is also a V-digraph: we uniquely have the function $V_{D_{\mathcal{S}},i}:G_{D_{\mathcal{S}},i} \rightarrow \mathbb{R}$ satisfying the relation ${\pi}_{2,1,i} {\mid}_{\overline{D_{\mathcal{S}}}}=V_{D_{\mathcal{S}},i} \circ q_{D_{\mathcal{S}},i}$. We consider the pair of the graph and the restriction of the function to the vertex set and call this V-graph the {\it Poincar\'e-Reeb V-digraph of $(\mathcal{S},D_{\mathcal{S}})$ for ${\pi}_{2,1,i}$}.

\subsection{The Poincar\'e-Reeb V-digraph of $(\mathcal{S},D_{\mathcal{S}})$ for ${\pi}_{2,1,i}$ with labels.}
We introduce labels to the V-digraphs. This is our new work.

Hereafter, we also assume fundamental knowledge on plane geometry. More precisely, geometry of straight lines and circles in the Euclidean place ${\mathbb{R}}^2$.

For subsets $A_1$ and $A_2$  in ${\mathbb{R}}^2$ such that the restrictions of ${\pi}_{2,1,1}$ (${\pi}_{2,1,2}$) are injective and the images in $\mathbb{R}$ are same, $A_1$ is {\it above} (resp. {\it in the right of}) $A_2$ if $(x_{1,1},x_{1,2,1}) \in A_1$ and $(x_{1,1},x_{1,2,2}) \in A_2$, then $x_{1,2,1}>x_{1,2,2}$ always holds (resp. if $(x_{1,1,1},x_{1,2}) \in A_1$ and $(x_{1,1,2},x_{1,2}) \in A_2$, then $x_{1,1,1}>x_{1,1,2}$ always holds).

\subsubsection{Labels to edges.}
 For each edge $e$ of the Poincar\'e-Reeb V-digraph of $(\mathcal{S},D_{\mathcal{S}})$ for ${\pi}_{2,1,i}$, we can consider ${(q_{D_{\mathcal{S}},i})}^{-1}(e)$ and the intersection ${(q_{D_{\mathcal{S}},i})}^{-1}(e) \bigcap \partial \overline{D_{\mathcal{S}}}$ consists of two smooth connected curves $C_{e,1}$ and $C_{e,2}$ the restrictions of ${\pi}_{2,1,i}$ to which are injective and have no singular points. We can also define $C_{e,2}$ is above $C_{e,1}$ in ${\mathbb{R}}^2$ in the case $i=1$ and we do so. We can also define $C_{e,2}$ is in the right of $C_{e,1}$ in ${\mathbb{R}}^2$ in the case $i=2$ and we do so. The two curves
$C_{e,1}$ and $C_{e,2}$ are also in the unique circles $S_{e,1}$ and $S_{e,2}$ of $\mathcal{S}$, respectively. The relation $S_{e,1}=S_{e,2}$ may hold.
We assign a pair of the following two sets of curves.
\begin{itemize}
\item The set of all $(j^{\prime},j^{\prime}+1) \frac{\pi}{4}$-arcs of $S_{e,1}$ intersecting $C_{e,1}$. 
\item The set of all $(j^{\prime},j^{\prime}+1) \frac{\pi}{4}$-arcs of $S_{e,2}$ intersecting $C_{e,2}$. 
\end{itemize}
In the case $S_{e,1}=S_{e,2}$, we have the following easily.
\begin{Prop}
\label{prop:1}
We concentrate on the case $S_{e,1}=S_{e,2}$. In the case $i=1$ {\rm (}$i=2${\rm )}, for elements in the first set here, it must hold that $4 \leq j^{\prime} \leq 7$ {\rm (}resp. $j^{\prime}=2,3,4,5${\rm )} and for elements in the second set here, it must hold that $0 \leq j^{\prime} \leq 3$ {\rm (}resp. $j^{\prime}=0,1,6,7${\rm )}.
\end{Prop}
\subsubsection{For labels to vertices of the graph.}
We prepare for labels to vertices of the graph. \\
\ \\ 
Case 1  For vertices of degree $1$. \\
 For each vertex $v_1$ of degree $1$ of the Poincar\'e-Reeb V-digraph of $(\mathcal{S}, D_{\mathcal{S}})$ for ${\pi}_{2,1,i}$, we investigate the resulting set ${(q_{D_{\mathcal{S}},i})}^{-1}(v_1)$. 

One of the cases is the case that this is a one-point set $x_{v_1}$ contained in exactly one circle $S_{v_1}$ from $\mathcal{S}$. In this case, this is the $(j,\frac{\pi}{4})$-pole with $j=0,4$ in the case $i=1$ and $j=2,6$ in the case $i=2$.

The remaining case is the case that this is a one-point set contained in exactly two circles $S_{v_{1,1}}$ and $S_{v_{1,2}}$ from $\mathcal{S}$. We consider a circle $S_{x_{v_1},\epsilon}$ of a sufficiently small radius $\epsilon>0$ centered at the point $x_{v_1}$.
For the tangent line of each circle $S_{v_{1,j^{\prime}}}$ at $x_{v_1}$, we have the unique point $x_{v_1,j^{\prime}}$ contained both in the small circle $S_{x_{v_1},\epsilon}$ and the tangent line which is either contained in the set $\overline{D_{\mathcal{S}}} \bigcap S_{x_{v_1},\epsilon}$ or closer to $\overline{D_{\mathcal{S}}} \bigcap S_{x_{v_1},\epsilon}$ under the induced metric of the set $S_{x_{v_1},\epsilon} \subset {\mathbb{R}}^2$. 

This argument is also important in the remaining cases. We first argue in a general situation. Hereafter, we use $v_0$ for a vertex of the general case and $x_{v_0} \in {(q_{D_{\mathcal{S}},i})}^{-1}(v_0)$ for a point contained in exactly two distinct circles from $\mathcal{S}$. We use $S_{v_{0,1}}$ and $S_{v_{0,2}}$ for these two circles. We consider the angle formed by the segments connecting the points $x_{v_0,j^{\prime}}$, denoting the points defined in the previously presented way, similarly, and $x_{v_0}$. The value of the angle must be greater than $0$ and smaller than $\pi$. For such arguments, see \cite{kitazawa3} for example.

We go back to the present case.

The segments must be in the (closed) region of the form $\{(x_1,x_2) \mid x_1 \geq {\rm (}\leq {\rm )}\ {\pi}_{2,1,1}(x_{v_1}) \}$ and at most one of the two segments may be in the boundary of the closed region in the case $i=1$. We have the following in addition easily.
\begin{Prop}
\label{prop:2}
In the case $i=1$ with the property that one of the two segments is in the boundary of the closed region $\{(x_1,x_2) \mid x_1 \geq {\rm (}\leq {\rm )}\ {\pi}_{2,1,1}(x_{v_1}) \}$ and that the two segments are in the closed region, the point $x_{v_1}$ must be the $(4,\frac{\pi}{4})$-pole {\rm (}resp. $(0,\frac{\pi}{4})$-pole{\rm )} of the circle $S_{v_{1,j^{\prime \prime}}}$ having the segment as a subset of the tangent line at $x_{v_1}$ and must not be the $(j,\frac{\pi}{4})$-pole of the circle $S_{v_{1,j^{\prime \prime \prime}}}$ having the remaining segment as a subset of the tangent line at $x_{v_1}$ with $j=0,4$. 
\end{Prop}
The segments must be in the (closed) region of the form $\{(x_1,x_2) \mid x_2 \geq {\rm (}\leq {\rm )}\ {\pi}_{2,1,2}(x_{v_1}) \}$ and at most one of the two segments may be in the boundary of the closed region in the case $i=2$.  We have the following in addition easily.
\begin{Prop}
\label{prop:3}
 In the case $i=2$ with the property that one of the two segments is in the boundary of the closed region $\{(x_1,x_2) \mid x_2 \geq {\rm (}\leq {\rm )}\ {\pi}_{2,1,2}(x_{v_1}) \}$ and that the two segments are in the closed region, the point $x_{v_1}$ must be the $(6,\frac{\pi}{4})$-pole {\rm (}resp. $(2,\frac{\pi}{4})$-pole{\rm )} of the circle $S_{v_{1,j^{\prime \prime}}}$ having the segment as a subset of the tangent line at $x_{v_1}$ and must not be the $(j,\frac{\pi}{4})$-pole of the circle $S_{v_{1,j^{\prime \prime \prime}}}$ having the remaining segment as a subset of the tangent line at $x_{v_1}$ with $j=2,6$. 
\end{Prop}
We discuss the general case again and abuse the previously used notation. We orient the segments as ones departing from the point $x_{v_0}$ to $x_{v_0,j^{\prime}}$ and we call such a pair of the segments a {\it tangent segment pair at} 
$(\overline{D_{\mathcal{S}}}, x_{v_0})$. We can represent the orientations of the segments by the form $(t_{v_0,1,1},t_{v_0,1,2})$ and $(t_{v_0,2,1},t_{v_0,2,1})$ respectively: each $(t_{v_0,j^{\prime},1},t_{v_0,j^{\prime},2})$ of the vectors corresponds to the tangent line at the point of the circle $S_{v_{0,j^{\prime}}}$. The signs of $t_{v_0,j^{\prime},1}$ and $t_{v_0,j^{\prime},2}$ are independent of tangent segment pairs at $(\overline{D_{\mathcal{S}}}, x_{v_0})$ (in the case the value of the component is $0$ for some tangent segment pair at $(\overline{D_{\mathcal{S}}}, x_{v_0})$ the values of the component are always $0$).

We discuss the present case again. We use the notation $v_1$ for the vertex and $x_{v_1}$ for the point of the one-point set again. We also abuse the notation in the previous paragraph.
\begin{Prop}
\label{prop:4}
In the case where $x_{v_1}$ is in exactly two circles $S_{v_{1,1}}$ and $S_{v_{1,2}}$ from $\mathcal{S}$, for the point the relation $t_{v_1,1,1}t_{v_1,2,1} \geq 0$ holds in the case $i=1$ and the relation $t_{v_1,1,2}t_{v_1,2,2} \geq 0$ holds in the case $i=2$. Furthermore, in this case, we have the following.
\begin{enumerate}
\item In the case $t_{v_1,1,1}t_{v_1,2,1}=0$ with $i=1$, at least one of  $t_{v_1,1,1}$ and $t_{v_1,2,1}$ is a non-zero number. For example, in the case $t_{v_1,1,1}>0$, $x_{v_1}$ must not be the $(j,\frac{\pi}{4})$-pole of the circle $S_{v_{1,1}}$ with $j=0,4$ and must be the $(4,\frac{\pi}{4})$-pole of the circle $S_{v_{1,2}}$ and in the case $t_{v_1,1,1}<0$, $x_{v_1}$ must not be the $(j,\frac{\pi}{4})$-pole of the circle $S_{v_{1,1}}$ with $j=0,4$ and must be the $(0,\frac{\pi}{4})$-pole of the circle $S_{v_{1,2}}$. 
\item In the case $t_{v_1,1,2}t_{v_1,2,2}=0$ with $i=2$, at least one of $t_{v_1,1,2}$ and $t_{v_1,2,2}$ is a non-zero number. For example, in the case $t_{v_1,1,2}>0$, $x_{v_1}$ must not be the $(j,\frac{\pi}{4})$-pole of the circle $S_{v_{1,1}}$ with $j=2,6$ and must be the $(6,\frac{\pi}{4})$-pole of the circle $S_{v_{1,2}}$ and in the case $t_{v_1,1,2}<0$, $x_{v_1}$ must not be the $(j,\frac{\pi}{4})$-pole of the circle $S_{v_{1,1}}$ with $j=2,6$ and must be the $(2,\frac{\pi}{4})$-pole of the circle $S_{v_{1,2}}$. 
\end{enumerate}
\end{Prop} 
\noindent Case 2  For vertices of degree $2$. \\
For each vertex $v_2$ of degree $2$ of the Poincar\'e-Reeb V-digraph of $(\mathcal{S}, D_{\mathcal{S}})$ for ${\pi}_{2,1,i}$, we investigate the resulting set ${(q_{D_{\mathcal{S}},i})}^{-1}(v_2)$. 

We can have the following by a small and routine argument.
\begin{Prop}
\label{prop:5}
The intersection ${(q_{D_{\mathcal{S}},i})}^{-1}(v_2) \bigcap \partial \overline{D_{\mathcal{S}}}$ is a two-point set.
\begin{enumerate}

\item Each point must be contained in exactly one or exactly two circles from $\mathcal{S}$. 
\item In the case a point of ${(q_{D_{\mathcal{S}},i})}^{-1}(v_2) \bigcap \partial \overline{D_{\mathcal{S}}}$ is in exactly one circle from $\mathcal{S}$, this must not be the $(j,\frac{\pi}{4})$-pole of the circle of $\mathcal{S}$ with $j=0, 4$ in the case $i=1$ {\rm (}resp. $j=2,6$ in the case $i=2${\rm )}. 
\item In addition, at least one point of ${(q_{D_{\mathcal{S}},i})}^{-1}(v_2) \bigcap \partial \overline{D_{\mathcal{S}}}$ must be contained in exactly two circles from $\mathcal{S}$.
\end{enumerate}
\end{Prop}

\begin{Prop}
\label{prop:6}
For a point in the two-point set  ${(q_{D_{\mathcal{S}},i})}^{-1}(v_2) \bigcap \partial \overline{D_{\mathcal{S}}}$ in exactly two circles $S_{v_{2,1}}$ and $S_{v_{2,2}}$ from $\mathcal{S}$, the following hold. We use $x_{v_2}$ for the point, We abuse the notation of this subsection.
\begin{enumerate}
\item The relation $t_{v_2,1,1}t_{v_2,2,1} \leq 0$ holds in the case $i=1$ and the relation $t_{v_2,1,2}t_{v_2,2,2} \leq 0$ holds in the case $i=2$. 
\item In the case $t_{v_2,1,1}t_{v_2,2,1}=0$ with $i=1$, at least one of  $t_{v_2,1,1}$ and $t_{v_2,2,1}$ is a non-zero number. For example, in the case $t_{v_2,1,1}>0$, $x_{v_2}$ must not be the $(j,\frac{\pi}{4})$-pole of the circle $S_{v_{2,1}}$ with $j=0,4$ and $x_{v_2}$ must be the $(0,\frac{\pi}{4})$-pole of the circle $S_{v_{2,2}}$ and in the case $t_{v_2,1,1}<0$, $x_{v_2}$ must not be the $(j,\frac{\pi}{4})$-pole of the circle $S_{v_{2,1}}$ with $j=0,4$ and $x_{v_2}$ must be the $(4,\frac{\pi}{4})$-pole of the circle $S_{v_{2,2}}$. 
\item In the case $t_{v_2,1,2}t_{v_2,2,2}=0$ with $i=2$, at least one of  $t_{v_2,1,2}$ and $t_{v_2,2,2}$ is a non-zero number. For example, in the case $t_{v_2,1,2}>0$, $x_{v_2}$ must not be the $(j,\frac{\pi}{4})$-pole of the circle $S_{v_{2,1}}$ with $j=2,6$ and $x_{v_2}$ must be the $(2,\frac{\pi}{4})$-pole of the circle $S_{v_{2,2}}$ and in the case $t_{v_2,1,2}<0$, $x_{v_2}$ must not be the $(j,\frac{\pi}{4})$-pole of the circle $S_{v_{2,1}}$ with $j=2,6$ and $x_{v_2}$ must be the $(6,\frac{\pi}{4})$-pole of the circle $S_{v_{2,2}}$. 
\end{enumerate}
\end{Prop}
\noindent Compare Proposition \ref{prop:6} to Proposition \ref{prop:4}, especially. \\
\ \\
\noindent Case 3  For remaining vertices.\\
For each vertex $v$ of degree at least $3$ of the Poincar\'e-Reeb V-digraph of $(\mathcal{S}, D_{\mathcal{S}})$ for ${\pi}_{2,1,i}$, we investigate the resulting set ${(q_{D_{\mathcal{S}},i})}^{-1}(v)$.  
Let $k_{v,+}$ ($k_{v,-}$) be the number of all oriented edges of the V-digraph departing from (resp.  entering) $v$. These numbers must be positive integers at least one of which must be greater than $1$ by the condition of the degree of the vertex. We present the following proposition with the word "convexity", important in the proof. We can also check this proposition immediately.  
\begin{Prop}
\label{prop:7}
The preimage
${(q_{D_{\mathcal{S}},i})}^{-1}(v)$ is diffeomorphic to $D^1$. 

First, the two points of the boundary are also in $\partial \overline{D_{\mathcal{S}}}$ and points as in Propositions \ref{prop:5} and \ref{prop:6} by considering the convexity of $\overline{D_{\mathcal{S}}}$ and both of them may be contained in single circles.

The intersection ${(q_{D_{\mathcal{S}},i})}^{-1}(v) \bigcap \partial \overline{D_{\mathcal{S}}}$ is a finite set consisting of exactly $k_{v,+}+k_{v,-}>2$ points. 
We concentrate on these points except the two points which are previously discussed.
By considering the convexity of $\overline{D_{\mathcal{S}}}$, $k_{v,+}-1$ points are $(4,\frac{\pi}{4})$-poles {\rm (}$(6,\frac{\pi}{4})$-poles{\rm )} of circles of $mathcal{S}$, $k_{v,-}-1$ points are $(0,\frac{\pi}{4})$-poles  {\rm (}$(2,\frac{\pi}{4})$-poles{\rm )} of circles of $\mathcal{S}$, and these $k_{v,+}+k_{v,-}-2$ circles are mutually distinct in the case $i=1$ {\rm (}$i=2${\rm )}.
\end{Prop}

 
\subsubsection{Our summary on the labels.}
By remembering Propositions of this subsection and surrounding arguments, we can check the following and we do not present its proof explicitly.
\begin{Thm}
\label{thm:1}
For the Poincar\'e-Reeb {\rm (}V-di{\rm )}graph of $(\mathcal{S},D_{\mathcal{S}})$ for ${\pi}_{2,1,i}$, we can define a map $l_{(\mathcal{S},D_{\mathcal{S}}),i}$ on the disjoint union of the vertex set and the edge set of this whose value at each element $s$ is a sequence of a finite length
of finite sets of $(j,j+1) \frac{\pi}{4}$-arcs intersecting ${q_{D_{\mathcal{S}},i}}^{-1}(s)$, and of some circles of $\mathcal{S}$, enjoying the following.
\begin{enumerate}
\item \label{thm:1.1} For each edge $e$, $l_{(\mathcal{S},D_{\mathcal{S}}),i}(e)$ is a sequence of length $2$ of finite sets whose sizes are at most $4$ of $(j,j+1) \frac{\pi}{4}$-arcs of some circles of $\mathcal{S}$ intersecting ${q_{D_{\mathcal{S}},i}}^{-1}(e)$. For each finite set, the circle containing the arcs is unique and let $S_{e,1}$ denote the circle containing the arcs in the first set and $S_{e,2}$ the circle containing the arcs in the second set with $S_{e,1}=S_{e,2}$ or $S_{e,2}$ being above {\rm (}in the right of {\rm )} $S_{e,1}$ in the case $i=1$ {\rm (}resp. $i=2${\rm )}. Here we have adopted the notation before.

In the case $i=1$, the union of all arcs of each finite set must be a connected smooth curve and the finite set must not contain the $(3,4) \frac{\pi}{4}$-arc and the $(4,5) \frac{\pi}{4}$-arc simultaneously or the $(7,8) \frac{\pi}{4}$-arc and the $(0,1) \frac{\pi}{4}$-arc of the uniquely defined circle simultaneously.  
In the case $i=1$ with $S_{e,1}=S_{e,2}$, the finite set of the first element must consist of  $(j,j+1) \frac{\pi}{4}$-arcs of $S_{e,1}$ with $j=4,5,6,7$ only and that of the second element must consist of  $(j,j+1) \frac{\pi}{4}$-arcs of $S_{e,2}$ with $j=0,1,2,3$ only. 
In the case $i=2$, the union of all arcs of each finite set must be a connected smooth curve  and the finite set must not contain the $(j,j+1) \frac{\pi}{4}$-arc and the $(j+1,j+2) \frac{\pi}{4}$-arc of the uniquely defined circle simultaneously for $j=1,5$. In the case $i=2$ with $S_{e,1}=S_{e,2}$, the finite set of the first element must consist of  $(j,j+1) \frac{\pi}{4}$-arcs of $S_{e,1}$ with $j=2,3,4,5$ only and that of the second element must consist of  $(j,j+1) \frac{\pi}{4}$-arcs of $S_{e,2}$ with $j=0,1,6,7$ only. 
\item  \label{thm:1.2} For each vertex $v_1$ of degree $1$, $l_{(\mathcal{S},D_{\mathcal{S}}),i}(v_1)$ is defined as in either of the following.
\begin{enumerate}
\item \label{thm:1.2.1} The set ${q_{D_{\mathcal{S}},i}}^{-1}(v_1)$ is a one-point set and contained as a subset of the unique circle of $\mathcal{S}$ and the single point is the $(j,\frac{\pi}{4})$-pole of the circle with $j=0,4$ in the case $i=1$ {\rm (}$j=2,6$ in the case $i=2${\rm )}. In this case, $l_{(\mathcal{S},D_{\mathcal{S}}),i}(v_1)$ is a sequence of length one and the finite set consists of the two $(j,j+1) \frac{\pi}{4}$-arcs of the circle of $\mathcal{S}$ containing the single point, or in other words $j=0,7$ or $j=3,4$ in the case $i=1$ and $j=1,2$ or $j=5,6$ in the case $i=2$.
\item \label{thm:1.2.2} The set ${q_{D_{\mathcal{S}},i}}^{-1}(v_1)$ is a one-point set and contained as subsets of the exactly two circles of $\mathcal{S}$, denoted by $S_{v_1,1}$ and $S_{v_1,2}$, as before. 
Here, we also use the notation before. The relation $t_{v_1,1,1}t_{v_1,2,1} \geq 0$ with $(t_{v_1,1,1},t_{v_1,2,1}) \neq (0,0)$ holds in the case $i=1$ and the relation $t_{v_1,1,2}t_{v_1,2,2} \geq 0$ with $(t_{v_1,1,2},t_{v_1,2,2}) \neq (0,0)$ holds in the case $i=2$. Either of the following cases is considered.
\begin{enumerate}
\item \label{thm:1.2.2.1} The relation $t_{v_1,1,1}t_{v_1,2,1}>0$ holds in the case $i=1$ and the relation $t_{v_1,1,2}t_{v_1,2,2}>0$ holds in the case $i=2$. Here, in the former case, $l_{(\mathcal{S},D_{\mathcal{S}}),i}(v_1)=l_{(\mathcal{S},D_{\mathcal{S}}),1}(v_1)$ is a sequence of length one and the finite set consists of one or two $(j,j+1) \frac{\pi}{4}$-arcs of the circle $S_{v_1,1}$ of $\mathcal{S}$ and one or two $(j,j+1) \frac{\pi}{4}$-arcs of the circle $S_{v_1,2}$ of $\mathcal{S}${\rm :} in the case two $(j,j+1) \frac{\pi}{4}$-arcs of the circle $S_{v_1,j^{\prime}}$ are contained in the finite set, then for the two $(j_{j^{\prime}},j_{j^{\prime}}+1) \frac{\pi}{4}$-arcs of the circle $S_{v_1,j^{\prime}}$, the pair of the two integers $j_{j^{\prime}}$ must be $(0,1)$, $(1,2)$, $(2,3)$, $(4,5)$, $(5,6)$, or $(6,7)$. Here, in the latter case, $l_{(\mathcal{S},D_{\mathcal{S}}),i}(v_1)=l_{(\mathcal{S},D_{\mathcal{S}}),2}(v_1)$ is a sequence of length one and the finite set consists of one or two $(j,j+1) \frac{\pi}{4}$-arcs of the circle $S_{v_1,1}$ of $\mathcal{S}$ and one or two $(j,j+1) \frac{\pi}{4}$-arcs of the circle $S_{v_1,2}$ of $\mathcal{S}${\rm :} in the case two $(j,j+1) \frac{\pi}{4}$-arcs of the circle $S_{v_1,j^{\prime}}$ are contained in the finite set, then for the two $(j_{j^{\prime}},j_{j^{\prime}}+1) \frac{\pi}{4}$-arcs of the circle $S_{v_1,j^{\prime}}$, the pair of the two integers $j_{j^{\prime}}$ must be $(0,1)$, $(2,3)$, $(3,4)$, $(4,5)$, $(6,7)$, or $(7,0)$. 
\item \label{thm:1.2.2.2} The relation $t_{v_1,1,1}t_{v_1,2,1}=0$ holds in the case $i=1$ and the relation $t_{v_1,1,2}t_{v_1,2,2}=0$ holds in the case $i=2$. For example, in the case $i=1$ with $t_{v_1,1,1}>0$ and $t_{v_1,2,1}=0$, $l_{(\mathcal{S},D_{\mathcal{S}}),i}(v_1)=l_{(\mathcal{S},D_{\mathcal{S}}),1}(v_1)$ is a sequence of length one and the finite set consists of one or two $(j,j+1) \frac{\pi}{4}$-arcs of the circle $S_{v_1,1}$ of $\mathcal{S}$ and the $(3,4) \frac{\pi}{4}$-arc and the $(4,5) \frac{\pi}{4}$-arc of the circle $S_{v_1,2}$ of $\mathcal{S}${\rm :} in the case two $(j,j+1) \frac{\pi}{4}$-arcs of the circle $S_{v_1,1}$ is contained in the finite set, then for the two $(j_{1},j_{1}+1) \frac{\pi}{4}$-arcs of the circle $S_{v_1,1}$, the pair of the two integers $j_{1}$ must be $(0,1)$, $(1,2)$, $(2,3)$, $(4,5)$, $(5,6)$, or $(6,7)$. For example, in the case $i=2$ with $t_{v_1,1,2}>0$ and $t_{v_1,2,2}=0$, $l_{(\mathcal{S},D_{\mathcal{S}}),i}(v_1)=l_{(\mathcal{S},D_{\mathcal{S}}),2}(v_1)$ is a sequence of length one and the finite set consists of one or two $(j,j+1) \frac{\pi}{4}$-arcs of the circle $S_{v_1,1}$ of $\mathcal{S}$ and the $(5,6) \frac{\pi}{4}$-arc and the $(6,7) \frac{\pi}{4}$-arc of the circle $S_{v_1,2}$ of $\mathcal{S}${\rm :} in the case two $(j,j+1) \frac{\pi}{4}$-arcs of the circle $S_{v_1,1}$ is contained in the finite set, then for the two $(j_{1},j_{1}+1) \frac{\pi}{4}$-arcs of the circle $S_{v_1,1}$, the pair of the two integers $j_{1}$ must be $(0,1)$, $(2,3)$, $(3,4)$, $(4,5)$, $(6,7)$, or $(7,0)$.
\end{enumerate}
\end{enumerate}
\item \label{thm:1.3}  For each vertex $v_2$ of degree $2$, $l_{(\mathcal{S},D_{\mathcal{S}}),i}(v_2)$ is a sequence of length two of the finite sets and each of the finite sets is as presented in the following. More precisely, ${q_{D_{\mathcal{S}},i}}^{-1}(v_2)$ must have exactly two points contained in some circle of $\mathcal{S}$, denoted by $x_{{v_2}_1}$ and $x_{{v_2}_2}$ with the rule that $x_{{v_2}_2}$ is above {\rm (}resp. in the right of{\rm )} $x_{{v_2}_1}$ in the case $i=1$ {\rm (}resp. $i=2${\rm )}. 
\begin{enumerate}
\item \label{thm:1.3.1} The first finite set in the sequence $l_{(\mathcal{S},D_{\mathcal{S}}),i}(v_2)$ depends on the type of the point $x_{{v_2}_1}$. In the case $x_{{v_2}_1}$ is contained in the unique circle from $\mathcal{S}$, then the finite set is chosen uniquely according to the rule in {\rm (}\ref{thm:1.1}{\rm )}. In the case $x_{{v_2}_1}$ is contained in exactly two circles $S_{{v_2}_1,1}$ and $S_{{v_2}_1,2}$ from $\mathcal{S}$, then the finite set is chosen uniquely according to the following rule where we abuse the notation similarly. First the we have
the relation $t_{{v_2}_1,1,1}t_{{v_2}_1,2,1} \leq 0$ with $(t_{{v_2}_1,1,1},t_{{v_2}_1,2,1}) \neq (0,0)$  in the case $i=1$ and the relation $t_{{v_2}_1,1,2}t_{{v_2}_1,2,2} \leq 0$ with $(t_{{v_2}_1,1,2},t_{{v_2}_1,2,2}) \neq (0,0)$ in the case $i=2$. We can consider either of the following.
\begin{enumerate}
\item \label{thm:1.3.1.1} 
The relation $t_{{v_2}_1,1,1}t_{{v_2}_1,2,1}<0$ holds in the case $i=1$ and the relation $t_{{v_2}_1,1,2}t_{{v_2}_1,2,2}<0$ holds in the case $i=2$. Here, in the former case, our finite set consists of one or two $(j,j+1) \frac{\pi}{4}$-arcs of the circle $S_{{v_2}_1,1}$ of $\mathcal{S}$ and one or two $(j,j+1) \frac{\pi}{4}$-arcs of the circle $S_{{v_2}_1,2}$ of $\mathcal{S}${\rm :} in the case two $(j,j+1) \frac{\pi}{4}$-arcs of the circle $S_{{v_2}_1,j^{\prime}}$ is contained in the finite set, then for the two $(j_{j^{\prime}},j_{j^{\prime}}+1) \frac{\pi}{4}$-arcs of the circle $S_{{v_2}_1,j^{\prime}}$, the pair of the two integers $j_{j^{\prime}}$ must be $(0,1)$, $(1,2)$, $(2,3)$, $(4,5)$, $(5,6)$, or $(6,7)$. Here, in the latter case, our finite set consists of one or two $(j,j+1) \frac{\pi}{4}$-arcs of the circle $S_{{v_2}_1,1}$ of $\mathcal{S}$ and one or two $(j,j+1) \frac{\pi}{4}$-arcs of the circle $S_{{v_2}_1,2}$ of $\mathcal{S}${\rm :} in the case two $(j,j+1) \frac{\pi}{4}$-arcs of the circle $S_{{v_2}_1,j^{\prime}}$ is contained in the finite set, then for the two $(j_{j^{\prime}},j_{j^{\prime}}+1) \frac{\pi}{4}$-arcs of the circle $S_{{v_2}_1,j^{\prime}}$, the pair of the two integers $j_{j^{\prime}}$ must be $(0,1)$, $(2,3)$, $(3,4)$, $(4,5)$, $(6,7)$, or $(7,0)$. 
\item \label{thm:1.3.1.2} The relation $t_{{v_2}_1,1,1}t_{{v_2}_1,2,1}=0$ with $(t_{{v_2}_1,1,1},t_{{v_2}_1,2,1}) \neq (0,0)$ holds in the case $i=1$ and the relation $t_{{v_2}_1,1,2}t_{{v_2}_1,2,2}=0$ holds with $(t_{{v_2}_1,1,2},t_{{v_2}_1,2,2}) \neq (0,0)$ holds in the case $i=2$.

In the case $i=1$ with $t_{{v_2}_1,1,1}>0$ and $t_{{v_2}_1,2,1}=0$, our finite set consists of one or two $(j,j+1) \frac{\pi}{4}$-arcs of the circle $S_{{v_2}_1,1}$ of $\mathcal{S}$ and the $(0,1) \frac{\pi}{4}$-arc and the $(7,8) \frac{\pi}{4}$-arc of the circle $S_{{v_2}_1,2}$ of $\mathcal{S}${\rm :} in the case two $(j,j+1) \frac{\pi}{4}$-arcs of the circle $S_{{v_2}_1,1}$ are contained in the finite set, then for the two $(j_{1},j_{1}+1) \frac{\pi}{4}$-arcs of the circle $S_{{v_2}_1,1}$, the pair of the two integers $j_{1}$ must be $(0,1)$, $(1,2)$, $(2,3)$, $(4,5)$, $(5,6)$, or $(6,7)$.

In the case $i=1$ with $t_{{v_2}_1,1,1}<0$ and $t_{{v_2}_1,2,1}=0$, our finite set consists of one or two $(j,j+1) \frac{\pi}{4}$-arcs of the circle $S_{{v_2}_1,1}$ of $\mathcal{S}$ and the $(3,4) \frac{\pi}{4}$-arc and the $(4,5) \frac{\pi}{4}$-arc of the circle $S_{{v_2}_1,2}$ of $\mathcal{S}${\rm :} in the case two $(j,j+1) \frac{\pi}{4}$-arcs of the circle $S_{{v_2}_1,1}$ are contained in the finite set, then for the two $(j_{1},j_{1}+1) \frac{\pi}{4}$-arcs of the circle $S_{{v_2}_1,1}$, the pair of the two integers $j_{1}$ must be $(0,1)$, $(1,2)$, $(2,3)$, $(4,5)$, $(5,6)$, or $(6,7)$.

In the case $i=2$ with $t_{{v_2}_1,1,2}>0$ and $t_{{v_2}_1,2,2}=0$, our finite set consists of one or two $(j,j+1) \frac{\pi}{4}$-arcs of the circle $S_{{v_2}_1,1}$ of $\mathcal{S}$ and the $(1,2) \frac{\pi}{4}$-arc and the $(2,3) \frac{\pi}{4}$-arc of the circle $S_{{v_2}_1,2}$ of $\mathcal{S}${\rm :} in the case two $(j,j+1) \frac{\pi}{4}$-arcs of the circle $S_{{v_2}_1,1}$ are contained in the finite set, then for the two $(j_{1},j_{1}+1) \frac{\pi}{4}$-arcs of the circle $S_{{v_2}_1,1}$, the pair of the two integers $j_{1}$ must be $(0,1)$, $(2,3)$, $(3,4)$, $(4,5)$, $(6,7)$, or $(7,0)$.

In the case $i=2$ with $t_{{v_2}_1,1,2}<0$ and $t_{{v_2}_1,2,2}=0$, our finite set consists of one or two $(j,j+1) \frac{\pi}{4}$-arcs of the circle $S_{{v_2}_1,1}$ of $\mathcal{S}$ and the $(5,6) \frac{\pi}{4}$-arc and the $(6,7) \frac{\pi}{4}$-arc of the circle $S_{{v_2}_1,2}$ of $\mathcal{S}${\rm :} in the case two $(j,j+1) \frac{\pi}{4}$-arcs of the circle $S_{{v_2}_1,1}$ are contained in the finite set, then for the two $(j_{1},j_{1}+1) \frac{\pi}{4}$-arcs of the circle $S_{{v_2}_1,1}$, the pair of the two integers $j_{1}$ must be $(0,1)$, $(2,3)$, $(3,4)$, $(4,5)$, $(6,7)$, or $(7,0)$. 
\end{enumerate}
\item \label{thm:1.3.2} The second finite set in the sequence $l_{(\mathcal{S},D_{\mathcal{S}}),i}(v_2)$ depends on the type of the point $x_{{v_2}_2}$ and we define the finite set in the way same as the case of the first finite set and $x_{{v_2}_1}$. We have another restriction that in the case the first finite set is as in {\rm (}\ref{thm:1.1}{\rm )}, the second finite set is not a finite set defined as in {\rm (}\ref{thm:1.1}{\rm )}.
\end{enumerate}
\item \label{thm:1.4}
For each vertex $v$ of degree at least $3$, $l_{(\mathcal{S},D_{\mathcal{S}}),i}(v)$
is defined as follows. Let $k_{v,+}$ {\rm (}$k_{v,-}${\rm )} denote the number of all oriented edges of the V-digraph departing from {\rm (}resp. entering{\rm )} $v$ as before. The set ${q_{D_{\mathcal{S}},i}}^{-1}(v)$ must have exactly $k_{v,+}+k_{v,-}>2$ points contained in some circles of $\mathcal{S}$, each of which is denoted by $x_{{v}_{0,j}}$ with integers $1 \leq j \leq k_{v,+}+k_{v,-}$ and the rule that $x_{{v}_{0,j+1}}$ is above {\rm (}in the right of{\rm )} $x_{{v}_{0,j}}$ for $1 \leq j \leq  k_{v,+}+k_{v,-}-1$ in the case $i=1$ {\rm (}resp. $i=2${\rm )}. 
\begin{enumerate}
\item \label{thm:1.4.1} The 1st and the {\rm (}$k_{v,+}+k_{v,-}${\rm )}-th finite sets in the sequence $l_{(\mathcal{S},D_{\mathcal{S}}),i}(v)$ depend on the types of the points $x_{{v}_{0,1}}$ and $x_{{v}_{0,k_{v,+}+k_{v,-}}}$, respectively. They are defined in the way presented in {\rm (}\ref{thm:1.3}{\rm )} uniquely. Different from the case {\rm (}\ref{thm:1.3}{\rm )} {\rm (}, presented in {\rm (}\ref{thm:1.3.2} explicitly{\rm )}{\rm )}, both finite sets may be defined in the way {\rm (}\ref{thm:1.1}{\rm )}.
\item \label{thm:1.4.2} The $k$-th finite set in the sequence $l_{(\mathcal{S},D_{\mathcal{S}}),i}(v)$ with $2 \leq k \leq k_{v,+}+k_{v,-}-1$ depends on the type of the point $x_{{v}_{0,k}}$. This set is the two-element set consisting of $(3,4) \frac{\pi}{4}$-arc and the $(4,5) \frac{\pi}{4}$-arc of a uniquely chosen circle or the two-element set consisting of $(0,1) \frac{\pi}{4}$-arc and the $(7,8) \frac{\pi}{4}$-arc of a uniquely chosen circle in the case $i=1${\rm :} $k_{v,+}-1$ of the finite sets are of the former type and $k_{v,-}-1$ of the finite sets are of the latter type. This set is the two-element set consisting of $(1,2) \frac{\pi}{4}$-arc and the $(2,3) \frac{\pi}{4}$-arc of a uniquely chosen circle or the two-element set consisting of $(5,6) \frac{\pi}{4}$-arc and the $(6,7) \frac{\pi}{4}$-arc of a uniquely chosen circle in the case $i=2${\rm :} $k_{v,+}-1$ of the finite sets are of the latter type and $k_{v,}-1$ of the finite sets are of the former type. For distinct finite sets in these $k_{v,+}+k_{v,-}-2 \geq 1$ sets here, these uniquely chosen $k_{v,+}+k_{v,-}-2 \geq 1$ circles of $\mathcal{S}$ are mutually distinct and these circles are different from the circles of $\mathcal{S}$ containing $x_{{v}_{0,1}}$ and $x_{{v}_{0,k_{v,+}+k_{v,-}}}$. 
\end{enumerate}
\item The present statement generalizes some statements of Proposition \ref{prop:1} and Theorem \ref{thm:1} {\rm (}\ref{thm:1.1}{\rm )}.

In the case $i=1$, if some $(j,j+1) \frac{\pi}{4}$-arcs of a single circle $S_s$ appear both in the $j_1$-th finite set ${\mathcal{F}}_{s,1}$ and the $j_2$-th finite set ${\mathcal{F}}_{s,2}$ with $j_1<j_2$ in the sequence $l_{(\mathcal{S},D_{\mathcal{S}}),i}(s)$ of finite sets, then these arcs in ${\mathcal{F}}_{s,1}$ must be $(j,j+1) \frac{\pi}{4}$-arcs of $S_s$ with $j=4,5,6,7$ only and those in ${\mathcal{F}}_{s,2}$ must be $(j,j+1) \frac{\pi}{4}$-arcs of $S_s$ with $j=0,1,2,3$ only.

In the case $i=2$, if some $(j,j+1) \frac{\pi}{4}$-arcs of a single circle $S_s$ appear both in the $j_1$-th finite set ${\mathcal{F}}_{s,1}$ and the $j_2$-th finite set ${\mathcal{F}}_{s,2}$ with $j_1<j_2$ in the sequence $l_{(\mathcal{S},D_{\mathcal{S}}),i}(s)$ of finite sets, then these arcs in ${\mathcal{F}}_{s,1}$ must be $(j,j+1) \frac{\pi}{4}$-arcs of $S_s$ with $j=2,3,4,5$ only and those in ${\mathcal{F}}_{s,2}$ must be $(j,j+1) \frac{\pi}{4}$-arcs of $S_s$ with $j=0,1,6,7$ only.

Furthermore, if these properties are enjoyed, then $j_1=1$ and $j_2$ is equal to the length of the sequence $l_{(\mathcal{S},D_{\mathcal{S}}),i}(s)$.
\end{enumerate}  
\end{Thm}
\begin{Ex}
\label{ex:1}
Let $\mathcal{S}=\{S^1\}$ and $D_{\mathcal{S}}=D^2$. The Poincar\'e-Reeb {\rm (}V-di{\rm )}graph of $(\mathcal{S},D_{\mathcal{S}})$ for ${\pi}_{2,1,i}$ consists of exactly one edge and two vertices.  
We consider the $i=1$ case.
For the edge $e$, $l_{(\mathcal{S},D_{\mathcal{S}}),1}(e)$ is a sequence of length $2$ such that the first finite set is the four element set consisting of $(j,j+1) \frac{\pi}{4}$-arcs of $S^1$ with $j=4,5,6,7$ and that the second finite set is the four element set consisting of the $(j,j+1) \frac{\pi}{4}$-arcs of $S^1$ with $j=0,1,2,3$. For the vertices $v_{e,1,1}$ and $v_{e,1,2}$ with $V_{D_{\mathcal{S}},1}(v_{e,1,1})<V_{D_{\mathcal{S}},1}(v_{e,1,2})$, $l_{(\mathcal{S},D_{\mathcal{S}}),1}(v_{e,1,1})$ is a sequence of length $1$ such that the finite set is the two element set consisting of the $(3,4) \frac{\pi}{4}$-arc and the $(4,5) \frac{\pi}{4}$-arc of $S^1$ and $l_{(\mathcal{S},D_{\mathcal{S}}),1}(v_{e,1,2})$ is a sequence of length $1$ such that the finite set is the two element set consisting of $(0,1) \frac{\pi}{4}$-arc and the $(7,8) \frac{\pi}{4}$-arc of $S^1$. 
  
\end{Ex}
\section{Local changes of Poincar\'e-Reeb V-digraphs of $(\mathcal{S},D_{\mathcal{S}})$ of a certain new type according to additions of circles.}
We concentrate on the Poincar\'e-Reeb V-digraphs of $(\mathcal{S},D_{\mathcal{S}})$ for ${\pi}_{2,1,1}$ and we can consider similar arguments for the Poincar\'e-Reeb V-digraphs of $(\mathcal{S},D_{\mathcal{S}})$ for ${\pi}_{2,1,2}$, thanks to to the symmetry. Our symmetry is presented explicitly in \cite[Proposition 2]{kitazawa3} for example and this is also very elementary.

A {\it chord} of a circle means a segment connecting two distinct points in the circle. For example, this is uniquely determined. The interior of the chord of the circle is in the interior of the circle. At each of the two points of the boundary of the chord of the circle the sum of tangent vector spaces of the chord of the circle and the circle is the tangent vector space of the plane there.

Here, we also need the notion of {\it parallel} subsets and "similarity" in plane geometry. The notion of {\it parallel} subsets and that of {\it parallel} transformation are also fundamental in Riemannian manifolds: they are defined by considering {\it Riemannian connections}. The case for ${\mathbb{R}}^n$ is also a specific case from the theory of Rimannian manifolds.
\begin{Thm}
\label{thm:2}
For a map $l_{(\mathcal{S},D_{\mathcal{S}}),1}$ in Theorem \ref{thm:1} and an edge $e_1$ of the graph, for each point $x_{j^{\prime},\mathcal{S}}$ in the interior of a $(j,j+1) \frac{\pi}{4}$-arc of some circle of $\mathcal{S}$ in a finite set in the sequence $l_{(\mathcal{S},D_{\mathcal{S}}),1}(e_1)$ , we can choose a corresponding circle $S_{x_{j^{\prime}},r_{j^{\prime}}}$ of radius $r_{j^{\prime}}>0$ centered at another point $x_{j^{\prime}} \in {\mathbb{R}}^2$ bounding the disk containing the point $x_{j^{\prime},\mathcal{S}}$ in its interior and at least one of the following two holds for local changes of the Poincar\'e-Reeb V-digraphs of $(\mathcal{S},D_{\mathcal{S}})$ for ${\pi}_{2,1,i}$ {\rm (}$i=1,2${\rm )}.
\begin{enumerate}
\item \label{thm:2.1} We can have all the following five cases by suitable choice of $x_{j^{\prime}} \in {\mathbb{R}}^2$ and $r_{j^{\prime}}>0$.
\begin{enumerate}
\item \label{thm:2.1.1}  For the Poincar\'e-Reeb V-digraph of $(\mathcal{S},D_{\mathcal{S}})$ for ${\pi}_{2,1,i}$ for each $i=1,2$, we add two vertices in the interior of $e_i$ in such a way that $e_2$ is a suitably chosen edge of the Poincar\'e-Reeb V-digraph of $(\mathcal{S},D_{\mathcal{S}})$ for ${\pi}_{2,1,2}$, that the function preserves the original orientations on edges and segments in the edges, that the restriction to the vertex set of the original V-digraph is same as the original function  ${l_{D_{\mathcal{S},i}}}$ on the vertex set of the original V-digraph $(W_{D_{\mathcal{S}},i},l_{D_{\mathcal{S},i}})$, and that the values at the new two vertices are sufficiently close. 
\item \label{thm:2.1.2}
For the Poincar\'e-Reeb V-digraph of $(\mathcal{S},D_{\mathcal{S}})$ for ${\pi}_{2,1,1}$, we add two vertices in the interior of $e_1$ and another new edge $e_{v,1}$ connecting one of the two new vertices $v_{e,1,1}$ and $v_{e,1,2}$ and another new vertex $v_{e,1}$ of degree $1$. We also define the new function ${l_{D_{\mathcal{S},1}}}^{\prime}$ on the vertex set of the new graph in the following way.
\begin{enumerate}
\item The restriction of ${l_{D_{\mathcal{S},1}}}^{\prime}$ to the vertex set of the original V-digraph is same as the original function  ${l_{D_{\mathcal{S},1}}}$ on the vertex set of the original V-digraph $(W_{D_{\mathcal{S}},1},l_{D_{\mathcal{S},1}})$.
\item The function ${l_{D_{\mathcal{S},1}}}^{\prime}$ preserves the original orientations on edges and segments in the edges. In the case the edge $e_{v,1}$ contains $v_{e,1,a}$ {\rm (}$a=1,2${\rm )}, the values ${l_{D_{\mathcal{S},1}}}^{\prime}(v_{e,1,a})$ and ${l_{D_{\mathcal{S},1}}}^{\prime}(v_{e,1})$ are chosen to be sufficiently close. 

\end{enumerate}
 
For the Poincar\'e-Reeb V-digraph of $(\mathcal{S},D_{\mathcal{S}})$ for ${\pi}_{2,1,2}$, we add two vertices in the interior of some edge $e_2$ of the graph. We also define the new function ${l_{D_{\mathcal{S},2}}}^{\prime}$ on the vertex set of the new graph in such a way that the function preserves the original orientations on edges and segments in the edges, that the restriction to the vertex set of the original V-digraph is same as the original function  ${l_{D_{\mathcal{S},2}}}$ on the vertex set of the original V-digraph $(W_{D_{\mathcal{S}},2},l_{D_{\mathcal{S},2}})$, and that the values at the new two vertices are sufficiently close. 
\item \label{thm:2.1.3}
 For the Poincar\'e-Reeb V-digraph of $(\mathcal{S},D_{\mathcal{S}})$ for ${\pi}_{2,1,i}$ for each $i=1,2$, we add two vertices in the interior of $e_i$ and another new edge $e_{v,i}$ connecting one of the two new vertices $v_{e,i,1}$ and $v_{e,i,2}$ and another new vertex $v_{e,i}$ of degree $1$ where $e_2$ is a suitably chosen edge of the Poincar\'e-Reeb V-digraph of $(\mathcal{S},D_{\mathcal{S}})$ for ${\pi}_{2,1,2}$. We also define the new function ${l_{D_{\mathcal{S},i}}}^{\prime}$ on the vertex set of the new graph in the following way.
\begin{enumerate}
\item The restriction of ${l_{D_{\mathcal{S},i}}}^{\prime}$ to the vertex set of the original V-digraph is same as the original function ${l_{D_{\mathcal{S},i}}}$ on the vertex set of the original V-digraph $(W_{D_{\mathcal{S}},i},l_{D_{\mathcal{S},i}})$.
\item The function ${l_{D_{\mathcal{S},i}}}^{\prime}$ preserves the original orientations on edges and segments in the edges. In the case the edge $e_{v,i}$ contains $v_{e,i,a}$ {\rm (}$a=1,2${\rm )}, the values ${l_{D_{\mathcal{S},i}}}^{\prime}(v_{e,i,a})$ and ${l_{D_{\mathcal{S},i}}}^{\prime}(v_{e})$ are chosen to be sufficiently close. 

\end{enumerate}
\item \label{thm:2.1.4}
 For the Poincar\'e-Reeb V-digraph of $(\mathcal{S},D_{\mathcal{S}})$ for ${\pi}_{2,1,1}$, we add two vertices $v_{e,1,a}$ {\rm (}$a=1,2${\rm )} in the interior of $e_1$ and two new edges $e_{v,1,a}$ connecting the new vertex $v_{e,1,a}$ and another new vertex $v_{e,1,a,+}$ {\rm (}$a=1,2${\rm )} of degree $1$. We also define the new function ${l_{D_{\mathcal{S},1}}}^{\prime}$ on the vertex set of the new graph in the following way.
\begin{enumerate}
\item The restriction of ${l_{D_{\mathcal{S},1}}}^{\prime}$ to the vertex set of the original V-digraph is same as the original function  ${l_{D_{\mathcal{S},1}}}$ on the vertex set of the original V-digraph $(W_{D_{\mathcal{S}},1},l_{D_{\mathcal{S},1}})$.
\item The function ${l_{D_{\mathcal{S},1}}}^{\prime}$ preserves the original orientations on edges and segments in the edges. The values ${l_{D_{\mathcal{S},1}}}^{\prime}(v_{e,1,1,+})$ and ${l_{D_{\mathcal{S},1}}}^{\prime}(v_{e,1,2,+})$ are chosen to be sufficiently close. The values ${l_{D_{\mathcal{S},1}}}^{\prime}(v_{e,1,a})$ and ${l_{D_{\mathcal{S},1}}}^{\prime}(v_{e,1,a,+})$ are also chosen to be sufficiently close for $a=1,2$. We can also do so that either of the inequality ${l_{D_{\mathcal{S},1}}}^{\prime}(v_{e,1,1})<{l_{D_{\mathcal{S},1}}}^{\prime}(v_{e,1,1,+})<{l_{D_{\mathcal{S},1}}}^{\prime}(v_{e,1,2,+})<{l_{D_{\mathcal{S},1}}}^{\prime}(v_{e,1,2})$ or ${l_{D_{\mathcal{S},1}}}^{\prime}(v_{e,1,2})<{l_{D_{\mathcal{S},1}}}^{\prime}(v_{e,1,2,+})<{l_{D_{\mathcal{S},1}}}^{\prime}(v_{e,1,1,+})<{l_{D_{\mathcal{S},1}}}^{\prime}(v_{e,1,1})$ holds. 

\end{enumerate}
For the Poincar\'e-Reeb V-digraph of $(\mathcal{S},D_{\mathcal{S}})$ for ${\pi}_{2,1,2}$, we add two vertices in the interior of some edge $e_2$ and another new edge $e_{v,2}$ connecting one of the two new vertices $v_{e,2,1}$ and $v_{e,2,2}$ and another new vertex $v_{e,2}$ of degree $1$ where $e_2$ is a suitably chosen edge of the Poincar\'e-Reeb V-digraph of $(\mathcal{S},D_{\mathcal{S}})$ for ${\pi}_{2,1,2}$. We also define the new function ${l_{D_{\mathcal{S}2}}}^{\prime}$ on the vertex set of the new graph in the following way.
\begin{enumerate}
\item The restriction of ${l_{D_{\mathcal{S},2}}}^{\prime}$ to the vertex set of the original V-digraph is same as the original function ${l_{D_{\mathcal{S},2}}}$ on the vertex set of the original V-digraph $(W_{D_{\mathcal{S}},2},l_{D_{\mathcal{S},2}})$.
\item The function ${l_{D_{\mathcal{S},2}}}^{\prime}$ preserves the original orientations on edges and segments in the edges. In the case the edge $e_{v,2}$ contains $v_{e,2,a}$ {\rm (}$a=1,2${\rm )}, the values ${l_{D_{\mathcal{S},2}}}^{\prime}(v_{e,2,a})$ and ${l_{D_{\mathcal{S},2}}}^{\prime}(v_{e})$ are chosen to be sufficiently close. 

\end{enumerate}
\item  \label{thm:2.1.5}
For the Poincar\'e-Reeb V-digraph of $(\mathcal{S},D_{\mathcal{S}})$ for ${\pi}_{2,1,i}$ for each $i=1,2$, we add two vertices $v_{e,i,a}$ {\rm (}$a=1,2${\rm )} in the interior of $e_i$ and two new edges $e_{v,i,a}$ connecting the new vertex $v_{e,i,a}$ and another new vertex $v_{e,i,a,+}$ {\rm (}$a=1,2${\rm )} of degree $1$ where $e_2$ is a suitably chosen edge of the Poincar\'e-Reeb V-digraph of $(\mathcal{S},D_{\mathcal{S}})$ for ${\pi}_{2,1,i}$. We also define the new function ${l_{D_{\mathcal{S},i}}}^{\prime}$ on the vertex set of the new graph in the following way.
\begin{enumerate}
\item The restriction of ${l_{D_{\mathcal{S},i}}}^{\prime}$ to the vertex set of the original V-digraph is same as the original function ${l_{D_{\mathcal{S},i}}}$ on the vertex set of the original V-digraph $(W_{D_{\mathcal{S}},i},l_{D_{\mathcal{S},i}})$.
\item The function ${l_{D_{\mathcal{S},i}}}^{\prime}$ preserves the original orientations on edges and segments in the edges. The values ${l_{D_{\mathcal{S},i}}}^{\prime}(v_{e,i,1,+})$ and ${l_{D_{\mathcal{S},i}}}^{\prime}(v_{e,i,2,+})$ are chosen to be sufficiently close. The values ${l_{D_{\mathcal{S},i}}}^{\prime}(v_{e,i,a})$ and ${l_{D_{\mathcal{S},i}}}^{\prime}(v_{e,i,a,+})$ are also chosen to be sufficiently close for $a=1,2$. We can also do so that for $i=1,2$, either of the inequalities ${l_{D_{\mathcal{S},i}}}^{\prime}(v_{e,1,1})<{l_{D_{\mathcal{S},i}}}^{\prime}(v_{e,i,1,+})<{l_{D_{\mathcal{S},i}}}^{\prime}(v_{e,i,2,+})<{l_{D_{\mathcal{S},i}}}^{\prime}(v_{e,i,2})$ or ${l_{D_{\mathcal{S},i}}}^{\prime}(v_{e,i,2})<{l_{D_{\mathcal{S},i}}}^{\prime}(v_{e,i,2,+})<{l_{D_{\mathcal{S},i}}}^{\prime}(v_{e,i,1,+})<{l_{D_{\mathcal{S},i}}}^{\prime}(v_{e,i,1})$ holds. 
\end{enumerate}
\end{enumerate}
\item \label{thm:2.2} In the previous case, the types of the changes for the Poincar\'e-Reeb V-digraphs of $(\mathcal{S},D_{\mathcal{S}})$ for ${\pi}_{2,1,1}$ and ${\pi}_{2,1,2}$ are reversed.
\end{enumerate}
\end{Thm}
\begin{proof}
We consider the case $x_{j^{\prime},\mathcal{S}}$ is in the interior of the $(2,3) \frac{\pi}{4}$-arc of a circle $S$ of $\mathcal{S}$ and of the second finite set in $l_{(\mathcal{S},D_{\mathcal{S}}),1}(e_1)$. In this case, we can consider a chord of the circle sufficiently close to $x_{j^{\prime},\mathcal{S}}$ and parallel to segments in the red colored straight lines in FIGURE \ref{fig:1}. We can choose the new circle $S_{x_{j^{\prime}},r_{j^{\prime}}}$ containing the two points of the boundary of the chord of the circle $S$ in such a way that the region surrounded by $S_{x_{j^{\prime}},r_{j^{\prime}}}$ and the chord of $S$ is similar to the region surrounded by the red colored straight line and the lower arc of the two arcs of $S^1$ obtained from the (red colored) straight line in FIGURE \ref{fig:1}. Each of the five cases corresponds to the location of the red colored straight line and the $(j,\frac{\pi}{4})$-poles of $S^1$ with $j=0,2,4,6$. The case (\ref{thm:2.1.1}) corresponds to the case where the red colored straight line is "beyond" these poles or the set of these four poles is above the corresponding four-point set in the corresponding red colored straight line. We consider a parallel transformation of the red colored straight line from the bottom to the top and we have corresponding cases (\ref{thm:2.1.2}), (\ref{thm:2.1.3}), and (\ref{thm:2.1.4}). Last, we have the case (\ref{thm:2.1.5}) and this corresponds to the case where the corresponding four-point set in the (red colored) straight line is above the set of these four poles. 
We have checked the case (\ref{thm:2.1}) in the statement.

We consider the case $x_{j^{\prime},\mathcal{S}}$ is in the interior of the $(3,4) \frac{\pi}{4}$-arc of the circle $S$ of $\mathcal{S}$ of a circle $S$ of $\mathcal{S}$ and of the second finite set in $l_{(\mathcal{S},D_{\mathcal{S}}),1}(e_1)$. We have the case (\ref{thm:2.2}) in the statement, thanks to the symmetry.

We consider the case $x_{j^{\prime},\mathcal{S}}$ is in the interior of the $(2,3) \frac{\pi}{4}$-arc or the $(3,4) \frac{\pi}{4}$-arc of the circle of $\mathcal{S}$ and of the first finite set in $l_{(\mathcal{S},D_{\mathcal{S}}),1}(e_1)$. Although we may need to choose the chord of the circle closer to $x_{j^{\prime},\mathcal{S}}$ beforehand,  considering that our boundary is locally "concave", we can show our desired fact similarly.

For the remaining cases, we can show similarly, thanks to the symmetry.
\begin{figure}
\includegraphics[width=80mm,height=80mm]{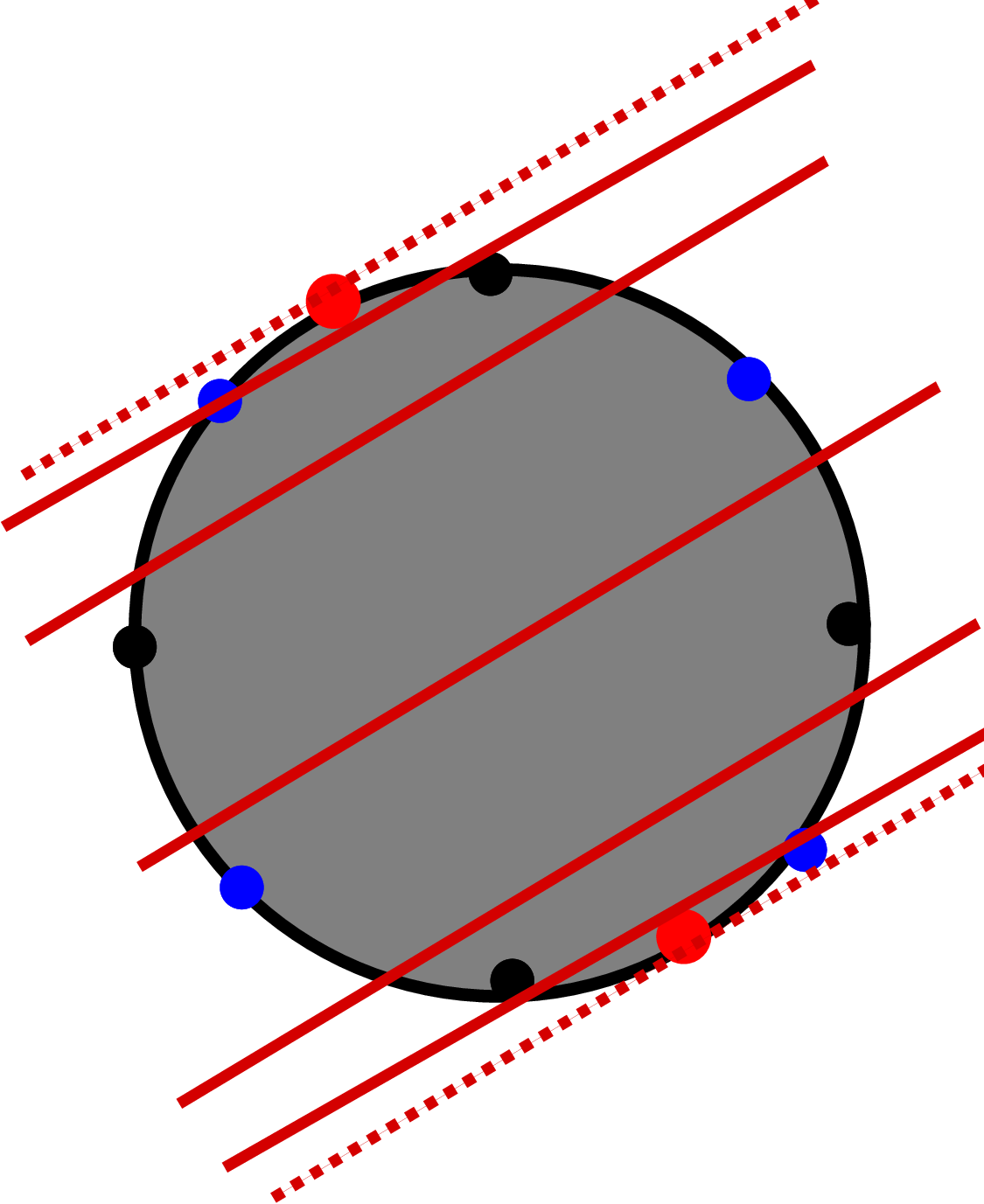}
\caption{The unit disk $D^2$ and the boundary $S^1$. Black dots show the $(j,\frac{\pi}{4})$-poles of the circle $S^1$ with $j=0,2,4,6$. Blue dots show the $(j,\frac{\pi}{4})$-poles of the circle $S^1$ with $j=1,3,5,7$. The (value of the) slopes of the red straight lines and straight dotted lines are same and they are positive and smaller than $1$. Each colored dotted line corresponds to each case of Theorem \ref{thm:2} (\ref{thm:2.1.1}), (\ref{thm:2.1.2}), (\ref{thm:2.1.3}), (\ref{thm:2.1.4}), and (\ref{thm:2.1.5}), from the bottom to the top. The red straight dotted lines are tangent lines to the points of the circle $S^1$.}
\label{fig:1}
\end{figure}
This completes the proof.
\end{proof}
The argument of the proof of Theorem \ref{thm:2} implies the following. We can easily check this and omit its proof.

\begin{Thm}
\label{thm:3}
In Theorem \ref{thm:2}, assume also that at least one finite set in the sequence $l_{(\mathcal{S},D_{\mathcal{S}}),1}(e)$ of length $2$ contains two $(j,j+1) \frac{\pi}{4}$-arcs of a circle from $\mathcal{S}$. In this case, we can have both cases {\rm (}\ref{thm:2.1}{\rm )} and {\rm (}\ref{thm:2.2}{\rm )} by choosing the points $x_{j^{\prime},\mathcal{S}}$ suitably in Theorem \ref{thm:2}.
\end{Thm}
We review related previous studies.
\begin{Def}[\cite{kitazawa3}]
If in Definition \ref{def:1}, in each step, the new circle $S_{x_{j^{\prime}},r_{j^{\prime}}}$ is chosen as a circle of a sufficiently small radius centered at a point in some given circle of $\mathcal{S}$, then the NI arrangement of circles is called an {\it MBCC arrangement}. 
\end{Def}
There the class of NI arrangements of circles has not been introduced. We review a simplest example of MBCC arrangements.
\begin{Ex}
\label{ex:2}
We consider the case of Example \ref{ex:1}. We choose a circle centered at a point of the interior of the union of the $(2,3)\frac{\pi}{4}$-arc and the $(3,4)\frac{\pi}{4}$-arc of the circle $S^1$ and of a sufficiently small radius.
This changes the Poincar\'e-Reeb V-digraphs of $(\mathcal{S},D_{\mathcal{S}})$ for ${\pi}_{2,1,i}$ into graphs with exactly five vertices and exactly four edges. More precisely, each graph is obtained from a graph having exactly four vertices and homeomorphic to $D^1$, by attaching the fourth edge to a vertex in the interior of the previous graph along the vertices.
\end{Ex}

Note that MBCC arrangements present the local changes of the graphs same as those for Theorem \ref{thm:2} (\ref{thm:2.1.3}).

\cite{kitazawa4} has introduced another class of NI arrangements of circles: an {\it SSC-NI arrangement}. However, this class is very complicated since we need to consider several cases and arguments for the definition. We do not review its rigorous definition. There the class of NI arrangements of circles has also been defined first. We review a simplest example of SSC-NI arrangements only.
\begin{Ex}
\label{ex:3}
We consider the case of Example \ref{ex:1}. We consider a chord of $S^1$ sufficiently close to a point in the interior of the union of the $(2,3)\frac{\pi}{4}$-arc and the $(3,4)\frac{\pi}{4}$-arc of the circle $S^1$. We put a circle $S_{x_{j^{\prime}},r_{j^{\prime}}}$ in such a way that the uniquely defined region surrounded by the new circle and the chord of $S^1$ and contained in $\overline{D_{\mathcal{S}}}$ is sufficiently small.
This does not change the underlying spaces of the Poincar\'e-Reeb V-digraphs of $(\mathcal{S},D_{\mathcal{S}})$ for ${\pi}_{2,1,i}$. This increases the number of vertices of each of these graphs by two. See also \cite[FIGURE 8]{kitazawa3} for example.
\end{Ex}
Note that SSC-NI arrangements present the local changes of the graphs same as those for Theorem \ref{thm:2} (\ref{thm:2.1.1}).

We explain a case for Theorems \ref{thm:2} and \ref{thm:3} by the following example.
\begin{Ex}
\label{ex:4}
Let ${S_{2}}^1$ be the circle centered at the origin $(0,0)$ and of radius $2$. Let $\mathcal{S}=\{S^1,{S_{2}}^1\}$ and $D_{\mathcal{S}}$ the region being also an open connected set of ${\mathbb{R}}^2$ and surrounded by $S^1$ and ${S_{2}}^1$. We can choose $e_1$ in Theorem \ref{thm:3} as the image $q_{D_{\mathcal{S}},1}(\{(x_1,x_2)\mid -2 < x_1 < -1\} \bigcap \overline{D_{\mathcal{S}}})$. The $(3,\frac{\pi}{4})$-pole of ${S_{2}}^1$ is $(-\sqrt{2},\sqrt{2})$. The edge $e_2$ in Theorem \ref{thm:3} can change according to the five cases in $q_{D_{\mathcal{S}},1}(\{(x_1,x_2)\mid -2 < x_1 < -1\} \bigcap \overline{D_{\mathcal{S}}})$. Let $x_{j^{\prime},\mathcal{S}}:=({x_{j^{\prime},\mathcal{S}}}_1,{x_{j^{\prime},\mathcal{S}}}_2)$ {\rm (}for the notation from Theorem \ref{thm:2}{\rm )}. 
\begin{itemize}
\item  $-2<{x_{j^{\prime},\mathcal{S}}}_1<-\sqrt{3}$. We can choose $e_2=q_{D_{\mathcal{S}},2}(\{(x_1,x_2)\mid x_1<0, -1<x_2<1\} \bigcap \overline{D_{\mathcal{S}}})$.
The case Theorem \ref{thm:2} (\ref{thm:2.2}) can occur.
\item  $-\sqrt{3}<{x_{j^{\prime},\mathcal{S}}}_1<-\sqrt{2}$ and $1<{x_{j^{\prime},\mathcal{S}}}_2<\sqrt{2}$. 
We can choose $e_2=q_{D_{\mathcal{S}},2}(\{(x_1,x_2)\mid x_2>1\} \bigcap \overline{D_{\mathcal{S}}})$.
The case Theorem \ref{thm:2} (\ref{thm:2.2}) can occur.
\item  $-\sqrt{3}<{x_{j^{\prime},\mathcal{S}}}_1<-\sqrt{2}$ and $-\sqrt{2}<{x_{j^{\prime},\mathcal{S}}}_2<-1$. We can choose $e_2=q_{D_{\mathcal{S}},2}(\{(x_1,x_2)\mid x_2<-1\} \bigcap \overline{D_{\mathcal{S}}})$. The case Theorem \ref{thm:2} (\ref{thm:2.2}) can occur.
\item  $-\sqrt{2}<{x_{j^{\prime},\mathcal{S}}}_1<-1$ and $\sqrt{2}<{x_{j^{\prime},\mathcal{S}}}_2<\sqrt{3}$. We can choose $e_2=q_{D_{\mathcal{S}},2}(\{(x_1,x_2)\mid x_2>1\} \bigcap \overline{D_{\mathcal{S}}})$. The case Theorem \ref{thm:2} (\ref{thm:2.1}) can occur.
\item  $-\sqrt{2}<{x_{j^{\prime},\mathcal{S}}}_1<-1$ and $-\sqrt{3}<{x_{j^{\prime},\mathcal{S}}}_2<-\sqrt{2}$.  We can choose $e_2=q_{D_{\mathcal{S}},2}(\{(x_1,x_2)\mid x_2<-1\} \bigcap \overline{D_{\mathcal{S}}})$. The case Theorem \ref{thm:2} (\ref{thm:2.1}) can occur.
\end{itemize}
\end{Ex}
\section{Additional remarks.}
\label{sec:5}
We present a natural problem.
\begin{Prob}
Formulate a category generalizing or respecting labeled Poincar\'e-Reeb V-digraphs of Theorem \ref{thm:1} and the third section. For each object in the category, can we realized it by explicitly obtaining $(\mathcal{S},D_{\mathcal{S}})$ up to isomorphisms in the category.
\end{Prob}

Our present study is, originally, motivated by the following.

We explain the {\it Reeb graph} of a smooth function shortly. The {\it Reeb graph} $W_c$ of a smooth function $c:X \rightarrow \mathbb{R}$ on a smooth manifold with no boundary is the graph whose underlying space is the space contracting connected components of preimages of single points to single points and regarded as the natural quotient space of $X$. We can define the quotient space as in the Poincar\'e-Reeb graph case. 
As \cite{reeb} shows, such objects are classical and have been fundamental and strong tools in representing the manifolds compactly. \cite{saeki1, saeki2} show that $W_c$ has the structure of a graph for a considerably wide class of smooth functions where a point is a vertex if and only if it represents a connected component containing some singular points of $c$.  

\begin{Prob}

For a given graph, can we reconstruct a real algebraic function on some connected component of the zero set of some polynomial map? 
\end{Prob}

This is originally a natural problem in differentiable functions and differential topology (\cite{sharko}). After several related studies, we omit related exposition on which in the present paper, the author has first considered the real algebraic case (\cite{kitazawa1}). For related exposition, see \cite{kitazawa1, kitazawa2, kitazawa3, kitazawa4} again. 

More precisely, the closure $\overline{D_{\mathcal{S}}}$ of the region $D_{\mathcal{S}}$ is the image of a natural real algebraic map of a certain class generalizing the class of the canonical projections ${\pi}_{m+1,2,i} {\mid}_{S^{m}}:S^m \rightarrow {\mathbb{R}}^n$ of the unit spheres (\cite{kitazawa1, kitazawa2}). Example \ref{ex:1} gives the canonical projection of the unit sphere ${\pi}_{m+1,2,i} {\mid}_{S^{m}}:S^m \rightarrow {\mathbb{R}}^2$.

For real algebraic geometry, see \cite{bochnakcosteroy, kollar} for example. 

Related to the fourth section, we present the following problem.
\begin{Prob}
Theorems \ref{thm:2} and \ref{thm:3} with Example \ref{ex:4} can yield another explicit class of NI arrangements of circles. This is new in starting from considering $(j,\frac{\pi}{4})$-poles of the circles explicitly. Generally, labels for Poincar\'e-Reeb V-digraphs are considered first, in Theorem \ref{thm:1}. Can we formulate other explicit classes of NI arrangements of circles? 
 \end{Prob} 
This is also presented in \cite{kitazawa3, kitazawa4} for example.
\section{Conflict of interest and Data availability.}
\noindent {\bf Conflict of interest.} \\
The author works at Institute of Mathematics for Industry: https://www.jgmi.kyushu-u.ac.jp/en/about/young-mentors/. This is closely related to our study. We thank them for their supports. The author is also a researcher at Osaka Central
Advanced Mathematical Institute (OCAMI researcher): this is supported by MEXT Promotion of Distinctive Joint Research Center Program JPMXP0723833165. Although he is not employed there, we also thank this. \\
\ \\
{\bf Data availability.} \\
Data related to our present study essentially are all in the present paper.

\end{document}